\documentclass{amsart}
\usepackage{mathrsfs,amssymb}
\newtheorem{theorem}{Theorem}[section]
\newtheorem{lemma}[theorem]{Lemma}
\newtheorem{proposition}[theorem]{Proposition}
\newtheorem{corollary}[theorem]{Corollary}
\newtheorem{definition}[theorem]{Definition}

\theoremstyle{definition}

\numberwithin{equation}{section}
\newcommand{\CC}{\mathbb{C}}
\newcommand{\Sch}{\mathscr{S}}
\newcommand{\prot}{{\lozenge}}
\newcommand{\RR}{\mathbb{R}}
\newcommand{\ZZ}{\mathbb{Z}}
\newcommand{\cO}{\mathcal{O}}
\newcommand{\cA}{\mathcal{A}}

\newcommand{\DD}{\mathbb{D}}
\newcommand{\dd}{\partial}
\newcommand{\cD}{\mathcal{D}}
\newcommand{\cZ}{\mathcal{Z}}

\newcommand{\bH}{\mathbf{H}}
\newcommand{\cH}{\mathcal{H}}
\renewcommand{\[}{\begin{equation}}
\renewcommand{\]}{\end{equation}}
\DeclareMathOperator{\const}{const}
\DeclareMathOperator{\ran}{ran}
\title[A product for discrete analytic functions]{On discrete analytic
functions: Products, Rational Functions, and some Associated
Reproducing Kernel Hilbert Spaces}

\begin{document}
\author[D. Alpay]{Daniel Alpay}
\address{(DA) Department of Mathematics \newline
Ben Gurion University of the Negev \newline P.O.B. 653, \newline
Be'er Sheva 84105, \newline ISRAEL} \email{dany@math.bgu.ac.il}
\author[P. Jorgensen]{Palle Jorgensen}
\address{(PJ)
Department of Mathematics\newline 14 MLH \newline The University
of Iowa, Iowa City,\newline IA 52242-1419 USA}
\email{palle-jorgensen@math.uiowa.edu}
\author[R. Seager]{Ron Seager}
\author[D. Volok]{Dan Volok}
\address{ (RS) and (DV) Mathematics Department\newline 138 Cardwell
Hall\newline Kansas State University,\newline Manhattan, KS 66506}
\email{danvolok@math.ksu.edu}

\thanks{D. Alpay thanks the
Earl Katz family for endowing the chair which supported his
research. The research of the authors was supported in part by the
Binational Science Foundation grant 2010117.}

\keywords{Discrete analytic functions, $2D$ lattice $\ZZ^2$,
reproducing kernel Hilbert space, Szeg\"o and Bergman,
multipliers, Cauchy integral representation, difference
operators, Lie algebra of operators, Fourier transform,
realizable linear  systems, expandable functions, rational
functions, Cauchy-Riemann equations, Cauchy-Kovalevskaya theorem,
Schur analysis, Fock space} \subjclass{Primary 30G25, 30H20,
32A26, 43A22 , 46E22, 46L08, 47B32, 47B39;  secondary: 20G43}
\maketitle

\begin{abstract}
We introduce a family of discrete analytic functions, called
expandable discrete analytic functions, which includes discrete
analytic polynomials, and define two products in this family. The
first one is defined in a way similar to the Cauchy-Kovalevskaya
product of hyperholomorphic functions, and allows us to define
rational discrete analytic functions. To define the second
product we need a new space of entire functions which is
contractively included in the Fock space. We study in this space
some counterparts of Schur analysis.
\end{abstract}
\tableofcontents
\section{Introduction}
\setcounter{equation}{0} In this paper, we explore a spectral
theoretic framework for representation of discrete analytic
functions. While the more familiar classical case of analyticity
plays an important role in such applications as the theory of
systems and their realizations, carrying over this to the case of
discrete analytic functions involves a number of operator- and
spectral theoretic subtleties; for example, we show that the
reproducing kernel, in the discrete case, behaves quite
differently from the case of the more familiar classical kernels
of Szeg\"o and Bergman. We introduce a reproducing kernel of
discrete analytic functions, which is naturally isomorphic to a
Hilbert space of entire functions contractively included in the
Fock space. The pointwise product of two discrete analytic
functions need not be discrete analytic, and we introduce two
products, each taking into account the specificities of discrete
analyticity.\\

The first product is determined by a solution to an extension
question for rational functions, extending from $\ZZ_+$ to the
right half-plane in the $2D$ lattice $\ZZ^2$. Our solution to the
extension problem leads to a new version of the "multiplication
operator" $\cZ$. We further prove that the product in $\cA$ will
be defined directly from $\cZ$. This in turn yields a
representation of the multiplier problem for the reproducing
kernel Hilbert space $\mathcal H$.  While it is possible to think
of the reproducing kernel Hilbert space $\cH$ as an extension of
one of the classical Beurling-Lax theory for Hardy space, the
case for discrete analytic function involves a new and different
spectral analysis, departing from the classical case in several
respects. For example, we show that the new multiplication
operator $\cZ$ is part of an infinite-dimensional non-Abelian Lie
algebra of operators acting on the reproducing kernel Hilbert
space $\cH$. With this, we are able to find the spectral type of
the operators described above.\\

The theory of  discrete analytic functions has drawn a lot of
attention recently, in part because of its connections with
electrical networks and random walks. In the case of functions
defined on the integer grid the notion of discrete analyticity
was introduced by J. Ferrand (Lelong) in \cite{MR0013411}:

\begin{definition}
\label{maindef}
A function $f:\ZZ^2\longrightarrow\CC$ is said to
be {\em discrete analytic}
 if
\[
\label{DCR} \forall (x,y)\in \ZZ^2,\qquad
\dfrac{f(x+1,y+1)-f(x,y)}{1+i}=\dfrac{f(x+1,y)-f(x,y+1)}{1-i}.\]
\end{definition}

 The properties of discrete analytic functions were
extensively investigated by R. J. Duffin in \cite{MR0078441}. In
this work it was shown that discrete analytic functions share many
important properties of the classical continuous analytic
functions in the complex domain, such as Cauchy integral
representation and the maximum modulus principle. More recently,
the notion of discrete analyticity and accompanying results were
extended by C. Mercat to the case of functions defined on
arbitrary graph embedded in an orientable surface; see
\cite{MR1824204}.\\

The concept of \underline{discrete} analyticity seems to cause
significant difficulties in the following regard: the pointwise
product of two discrete analytic functions is not necessarily
discrete analytic. For example, the functions $z:=x+iy$ and $z^2$
are discrete analytic in the sense of Definition \ref{maindef},
but $z^3$ is not. Thus a natural question arises, how to describe
all complex polynomials in two variables $x,y$ whose restriction
to the integer grid $\ZZ^2$ is discrete analytic, and, more
generally, rational discrete analytic functions. This problem was
originally considered by R.~Isaacs, using a definition of discrete
analyticity (the so-called monodiffricity), which is
algebraically simpler than Definition \ref{maindef}. In
\cite{MR0052526} R. Isaacs has posed a conjecture that {\em all
monodiffric rational functions are polynomials.} This conjecture
was disproved by C. Harman in \cite{MR0361111}, where an explicit
example of a non-polynomial
monodiffric function, rational in one quadrant, was constructed.\\

The results of R. Isaacs and C. Harman suggest that in the setting
of discrete analytic functions the notion of rationality based on
the pointwise product is not a suitable one. In order to
introduce a class of rational discrete analytic functions, which
would be sufficiently rich for applications, one needs a suitable
definition of the product. This is one of the main objective of
the present paper to introduce two products in the setting of
discrete analytic functions.\\

{\bf Organization:} The paper is organized as follows. Sections
\ref{sec2} through \ref{sec4} cover our preparation of the
discrete framework: analysis and tools. In Definition
\ref{def21}, the notion of "discrete analyticity" makes a key
link between the representation in the two integral variables
($x$, $y$) in the $2$-lattice $\mathbb Z^2$ , thus making precise
the interaction between the two integral variables $x$ and $y$
implied by analyticity. The notion of analyticity in the discrete
case is a basic rule (Definition \ref{maindef}) from which one
makes precise contour-summations around closed loops in $\mathbb
Z^2$.  In section \ref{sec2}, we introduce a basis system of
polynomials (which will appear to be restriction to the positive
real axis of discrete analytic polynomials $\zeta_n$ defined in
section \ref{sec5}), see equation \eqref{poly}. We further
introduce a discrete Fourier transform for functions of $(x,y)$
in the $2$-lattice $\mathbb Z^2$, and in the right half-plane
$\mathbb H_+=\mathbb Z_+\times \ZZ$ in $\mathbb Z^2$. The Fourier
representation in $\mathbb H_+$   is then used in sections
\ref{sec3} and \ref{sec4}; where we study extensions from $\ZZ$
to $\ZZ^2$, and from $\ZZ_+$ to $\mathbb H_+$. We begin our
analysis in section \ref{sec4} with some lemmas for the
polynomial case. Theorem \ref{maindef2} offers a discrete version
of the Cauchy-Riemann equations; we show that the discrete
analytic functions are defined as the kernel of a del-bar operator
$\overline{\mathcal D}$; to be studied in detail in section
\ref{sec7}, as part of a Lie algebra representation. In Theorem
\ref{t1} we show that every polynomial function on $\ZZ$ has a
unique discrete analytic extension to $\ZZ^2$. Sections
\ref{sec5} and \ref{sec6} deal with expandable functions
(Definition \ref{5.1}), and section \ref{sec7} rational discrete
analytic functions. The expandable functions are defined from a
basis system of discrete analytic polynomials $\zeta_n$ from
section \ref{sec2}, and a certain Cauchy-estimate, equation
\eqref{esti}. We shall need two products defined on expandable
functions, the first is our Cauchy-Kovalesvskaya product (in
section \ref{sec6}), and the second (section \ref{next}) is
defined on algebra generated by the discrete analytic polynomials
$\zeta_n$. Its study makes use of realizations from linear
systems theory. The definition of the Cauchy-Kovalesvskaya
product relies on uniqueness of extensions for expandable
functions (Corollary \ref{c6.3}). Hence it is defined first for
expandable functions, and then subsequently enlarged; first to
the discrete analytic rational functions (section \ref{sec7}),
and then to a new reproducing kernel Hilbert space in section
\ref{sec8}. The latter reproducing kernel Hilbert space  has its
kernel defined from the discrete analytic polynomials $\zeta_n$;
see \eqref{rkda}. The study of the reproducing kernel Hilbert
space  in turns involves such tools from analysis as
representations of Lie algebras (Theorem \ref{tm34}), and of
$C^*$-algebras (Theorem \ref{tm84}).\\

\section{Polynomials and rational functions on the set of
integers} \label{sec2} \setcounter{equation}{0}

In what follows, $\Omega$ stands for one of two sets: $\ZZ$ or
$\ZZ_+$, and $x^{[n]}$ denotes the polynomial of degree $n$
defined by
\[
\label{poly}
x^{[n]}:=\prod_{j=0}^{n-1}(x-j)\] (if $n=0,$
$x^{[0]}:=1$).\\

They have the generating function
\[
(1+t)^x=\sum_{k=0}^x x^{[k]}\frac{t^k}{k!},
\]
and their discrete analytic extensions $\zeta_n(x,y)$ are studied
in section \ref{sec5}.\\

The purpose of this section is to prove (see Theorem \ref{lim1}
below) that any rational function $f\,\,:\,\,\mathbb
Z_+\,\,\longrightarrow\,\,\mathbb C$ (see Definition
\ref{defrat1}) has a unique representation
\[
f(x)=\sum_{n=1}^\infty \widehat{f}(n)x^{[n]},
\]
where
\[
\limsup_{n\rightarrow\infty}\left(n!|\widehat{f}(n)|\right)^{\frac{1}{n}}\le
1.
\]
\begin{definition}
\label{def21}
The linear {\em difference operator} $\delta$ on the
space of functions $f:\Omega\longrightarrow\CC$  is defined by
$$(\delta f)(x)=f(x+1)-f(x),\quad x\in\Omega.$$
\end{definition}

\begin{proposition}\label{tay1}
Let $f:\ZZ_+\longrightarrow \CC$. Then, for every $x\in\ZZ_+$ the
series $$\check f(x):=\sum_{n\in\ZZ_+}f(n)x^{[n]}$$ has a finite
number of non zero terms.
\end{proposition}
\begin{proof}
In view of \eqref{poly},
$$\forall x\in\ZZ_+,\,\,\forall n\in\ZZ_+,\qquad
x<n\Longrightarrow x^{[n]}=0.$$

Therefore, for every $x\in\ZZ_+$ the series $\check f(x)$
contains at most $x+1$ non-zero terms.
\end{proof}

\begin{proposition}
\label{tay2} Let $f:\ZZ_+\longrightarrow \CC$. Then there exists a
unique function $\hat f:\ZZ_+\longrightarrow \CC$  such that
\[\label{tay}\forall x\in\ZZ_+,\qquad f(x)=\sum_{n\in\ZZ_+} \hat f(n)x^{[n]}.\]  The
function $\hat f(n)$ is given by
\[\label{fou1}\hat f(n)=\dfrac{(\delta^nf)(0)}{n!}.\]
\end{proposition}
\begin{proof} To show the uniqueness of $\hat f(n)$, assume that \eqref{tay} holds
 for some function $\hat f(n)$
and apply  the identity
\[\label{hu2}\delta x^{[n]}=nx^{[n-1]}\]
repeatedly to obtain \eqref{fou1}.\\

Next, let the function $\hat f(n)$ be defined by \eqref{fou1}.
Then, according to Proposition \ref{tay1}, the series
$$\tilde f(x):=\sum_{n\in\ZZ_+} \hat f(n)x^{[n]}$$
converges absolutely for every $x\in\ZZ_+$ and defines a function
$\tilde f:\ZZ_+\longrightarrow\CC.$ Hence \eqref{fou1} holds with
$\tilde f$ replacing $f$ and
$$\forall n\in\ZZ_+,\qquad (\delta^nf)(0)=n! \hat f(n)=(\delta^n\tilde f)(0).$$
 Now one can verify that
\[\label{hu1}\forall x\in\ZZ_+,\,\,\forall n\in\ZZ_+,\qquad
(\delta^nf)(x)=(\delta^n\tilde f)(x)
\]
by induction on $x$: If
$$ \forall n\in\ZZ_+, \qquad (\delta^nf)(x)=(\delta^n\tilde f)(x)$$
then
\begin{multline*}
\forall n\in\ZZ_+, \qquad
(\delta^nf)(x+1)=(\delta^nf)(x)+(\delta^{n+1}f)(x)\\=
(\delta^n\tilde f)(x)+(\delta^{n+1}\tilde f)(x) =(\delta^n\tilde
f)(x+1).\end{multline*} It remains to set $n=0$ in \eqref{hu1} to
obtain \eqref{tay}.
\end{proof}

\begin{definition}
Let $f:\ZZ_+\longrightarrow \CC$. Then the function $\hat
f:\ZZ_+\longrightarrow\CC,$ defined by \eqref{fou1}, is said to be
the Fourier transform of $f$.
\end{definition}

Suppose  that a function $f:\ZZ\longrightarrow\CC$ is such that
$$\forall x\in\ZZ,\qquad f(x)=p(x),$$
where $p(x)$ is a polynomial with complex coefficients. Then such
a polynomial $p(x)$ is unique, and we shall call the function $f$
itself a polynomial. If $f(x)\not\equiv 0,$ the degree of $f(x)$
is the same as the degree of $p(x)$.

\begin{proposition}\label{p01}
Let $f:\ZZ\longrightarrow\CC$ be a polynomial, and let $\hat f
:\ZZ_+\longrightarrow\CC$ be the Fourier transform of the
restriction $f_{|_{\ZZ_+}}$. Then the function $\hat f$ has a
finite support and
$$\forall x\in\ZZ,\qquad f(x)=\sum_{n\in\ZZ_+}\hat f(n)x^{[n]}.$$
\end{proposition}
\begin{proof}
Note that
$$\deg(\delta f)=\deg(f)-1$$
(if $f=\const,$ $\delta f=0$). Hence
$$\forall n\in\ZZ_+,\qquad n>\deg(f)\Longrightarrow \delta^nf=0.$$
Since $$(\delta f)_{|_{\ZZ_+}}=\delta(f_{|_{\ZZ_+}}),$$
\eqref{fou1} implies that
$$\forall n>\deg(f),\qquad \hat f(n)=0.$$
It follows that
$$\sum_{n\in\ZZ_+}\hat f(n)x^{[n]}$$ is, in fact, a polynomial,
which coincides with the polynomial $f(x)$ on $\ZZ_+$ and hence on
$\ZZ$.
\end{proof}
It follows from Proposition \ref{p01} and identity \eqref{hu2}
that if $f:\ZZ\longrightarrow\CC$ is a polynomial then so is
$\delta f$. A converse statement can be formulated as follows:
\begin{proposition}\label{p03}
Let $f:\ZZ\longrightarrow\CC$ be a polynomial. Then there exists a
polynomial $g:\ZZ\longrightarrow\CC$ such that $f(x)\equiv (\delta
g)(x).$  If  $f(x)\not\equiv0$ then $\deg(g)=\deg(f)+1.$
\end{proposition}
\begin{proof} By Proposition
\ref{p01},
$$ f(x)=\sum_{n\in\ZZ_+}\hat f(n)x^{[n]},$$ where
$\hat f:\ZZ_+\longrightarrow\CC$ is a function with finite
support.
 Consider the
polynomial
$$g(x)=\sum_{n\in\ZZ_+}\dfrac{\hat{f}(n)}{n+1}x^{[n+1]},$$
then, in view of \eqref{hu2},
$$\forall x\in\ZZ,\qquad (\delta g)(x)=\sum_{n\in\ZZ_+}{\hat{f}(n)}x^{[n]}=f(x).$$
\end{proof}

\begin{definition}
\label{defrat1} A function $f:\ZZ_+\longrightarrow \CC$ is said
to be {\em rational} if there exist polynomials
$p,q:\ZZ_+\longrightarrow\CC$ such that
$$\forall x\in \ZZ_+,\qquad q(x)\not=0,$$
and
$$\forall x\in \ZZ_+,\qquad f(x)=\dfrac{p(x)}{q(x)}.$$
\end{definition}

\begin{proposition}\label{ABC}
Let $f:\ZZ_+\longrightarrow \CC$ be given. Then $f(x)$ is a
rational function if and only if there exist a polynomial $p(x)$
and matrices $A,B,C$ such that
\[
\label{kaas} \sigma(A)\cap\ZZ_+=\emptyset,
\]
 and
\begin{equation}
\label{eqreal} f(x)=p(x)+C(xI-A)^{-1}B.
\end{equation}
\end{proposition}

\begin{proof} We first recall that a matrix-valued rational function of a
complex variable can always be written in the form
\[
r(z)=p(z)+C(zI-A)^{-1}B,
\]
where the matrix-valued polynomial $p$ takes care of the pole at
infinity, and $A,B$ and $C$ are matrices of appropriate sizes.
Furthermore, when the dimension of $A$ is minimal, the (finite)
poles of $r$ coincide with the spectrum of $A$; see \cite{bgk1,
MR0255260}. Here, we consider complex-valued functions, and
thus $C$ and $B$ are respectively a row and column vector.\\

Now, by Definition \ref{defrat1}, $f$ is a rational function if
and only if it is the restriction on $\ZZ_+$ of a rational
function of a complex variable, with no poles on $\ZZ_+$, that is
if and only if it can be written as \eqref{eqreal} with the
matrix $A$ satisfying furthermore \eqref{kaas}.
\end{proof}

\begin{theorem}\label{lim1}
Let $f:\ZZ_+\longrightarrow \CC$ be rational, and let $\hat
{f}:\ZZ_+\longrightarrow \CC$ be the Fourier transform of $f$.
Then
$$\limsup_{n\rightarrow\infty}(|\hat f(n)|n!)^{1/n}\leq 1.$$
\end{theorem}
\begin{proof}
According to Proposition \ref{ABC}, there exist a polynomial
$p(x)$ and matrices $A,B,C$ such that
$\sigma(A)\cap\ZZ_+=\emptyset$ and
$$f(x)=p(x)+C(xI-A)^{-1}B.$$ Hence, for $n>\deg(p)$,
$$|\hat{f}(n)n!|=|(\delta^nf)(0)|\leq\|C\|\cdot\|B\|\cdot\|A^{-1}\|\cdot
\prod_{j=1}^n\|\left(I-\dfrac{1}{n}A\right)^{-1}\|.$$ Since
$$\lim_{n\rightarrow\infty}\|\left(I-\dfrac{1}{n}A\right)^{-1}\|=1,$$
$$\forall\epsilon>0,\,\,\exists M,N\in\ZZ_+,\,\forall
n\in\ZZ_+\quad:\quad n\geq N\Longrightarrow|\hat{f}(n)n!|\leq
M(1+\epsilon)^n,$$ and the conclusion follows.
\end{proof}

\section{Discrete polynomials of two variables}
\setcounter{equation}{0}
\label{sec3}
\begin{definition}
The linear {\em difference operators} $\delta_x,\delta_y$ on the
space of functions $f:\Omega_1\times\Omega_2\longrightarrow\CC$
are defined by
$$(\delta_xf)(x,y):=f(x+1,y)-f(x,y),\quad (\delta_yf)(x,y):=f(x,y+1)-f(x,y).$$
\end{definition}
Note that the difference operators $\delta_x$ and $\delta_y$
commute:
\[\label{comm}(\delta_x\delta_yf)(x,y)=
(\delta_y\delta_xf)(x,y)=f(x+1,y+1)-f(x,y+1)-f(x+1,y)+f(x,y).
\]

\begin{proposition}\label{tay11}
Let $f:\ZZ_+^2\longrightarrow \CC$. Then for every
$(x,y)\in\ZZ_+^2$, the series
\[
\label{transform} \check
f(x,y):=\sum_{(m,n)\in\ZZ_+^2}f(m,n)x^{[m]}y^{[n]}
\]
contains finitely many non-zero terms.
\end{proposition}
\begin{proof}
In view of \eqref{poly},
$$\forall x\in\ZZ_+,\,\forall
n\in\ZZ_+,\,\quad x<n\Longrightarrow x^{[n]}=0.$$  Therefore, for
every $(x,y)\in\ZZ_+^2$ the series $\check f(x)$ contains at most
$(x+1)(y+1)$ non-zero terms.
\end{proof}

Formula \eqref{transform} can be viewed as a transform of a
discrete function. The inverse transform is calculated in the
next proposition.

\begin{proposition}\label{tay12}
Let $f:\ZZ_+^2\longrightarrow \CC$. Then there exists a unique
function $\hat f:\ZZ_+^2\longrightarrow \CC$  such that
\[\label{tay02}\forall (x,y)\in\ZZ_+^2,
\qquad f(x,y)=\sum_{(m,n)\in\ZZ_+^2} \hat f(m,n)x^{[m]}y^{[n]}.
\]
The function $\hat f(m,n)$ is given by
\[
\label{fou11} \hat
f(m,n)=\dfrac{(\delta_x^m\delta_y^nf)(0,0)}{m!n!}, \quad
(m,n)\in\ZZ_+^2.\]
\end{proposition}
\begin{proof} First, fix $x\in\ZZ_+$ and consider the function
$f_x:\ZZ_+\longrightarrow\CC$ given by $$f_x(y)=f(x,y),\quad
y\in\ZZ_+.$$ Then, according to Proposition \ref{tay2}, there is a
unique function $\hat{ f_x}:\ZZ_+\longrightarrow\CC$ such that
$$\forall y\in\ZZ_+,\qquad f_x(y)=\sum_{n\in\ZZ_+}\hat{f_x}(n) y^{[n]};$$
the function $\hat{f_x}(n)$ is given by
$$\hat{f_x}(n)=\dfrac{(\delta_y^n f)(x,0)}{n!},\quad n\in\ZZ_+.$$
Next, fix $n\in\ZZ_+$ and  consider the function
$g_n:\ZZ_+\longrightarrow \CC$ given by
$$g_n(x)=\hat{f_x}(n),\quad x\in\ZZ_+.$$ By the same Proposition
\ref{tay2}, there is a unique function
$\hat{g_n}:\ZZ_+\longrightarrow \CC$ such that
$$\forall x\in\ZZ_+,\qquad g_n(x)=\sum_{m\in\ZZ_+}\hat{g_n}(m) x^{[m]};$$
the function $\hat{g_n}(m)$ is given by
$$\hat{g_n}(m)=\dfrac{(\delta_x^m\delta_y^n f)(0,0)}{m!n!},\quad m\in\ZZ_+.$$
Thus
$$\forall (x,y)\in\ZZ_+^2,\qquad f(x,y)=\sum_{(m,n)
\in\ZZ_+^2}\hat{g_n}(m) x^{[m]}y^{[n]}.$$ It remains to set $\hat
f(m,n)=\hat{g_n}(m).$
\end{proof}

\begin{definition}
Let $f:\ZZ_+^2\longrightarrow \CC$. Then the function $\hat
f:\ZZ_+^2\longrightarrow\CC,$ defined by \eqref{fou11}, is said to
be the Fourier transform of $f$.
\end{definition}

\begin{theorem}
\label{tautology} Let $p(z,w)$ be a complex polynomial in two
variables. Then, $p\big|_{\ZZ^2}=0$ if and only if $p\equiv 0$.
\end{theorem}

\begin{proof}
Write
\[
p(z,w)=\sum_{n=0}^N p_n(z)w^n,
\]
where the $p_n$ are polynomials in $z$. The equations
$p(z,w)\equiv 0$ for $w=0,1,\ldots N$ lead to (using a
Vandermonde determinant) that $p_0(z),p_1(z),\ldots, p_N(z)$
vanish on $\ZZ$ and hence identically.
\end{proof}

Suppose  that a function $f:\ZZ^2\longrightarrow\CC$ is such that
$$\forall (x,y)\in\ZZ^2,\qquad f(x,y)=p(x,y),$$
where $p(x,y)$ is a polynomial with complex coefficients. Then in
view of Theorem \ref{tautology}, such a polynomial $p(x,y)$ is
unique, and we shall call the function $f$ itself a polynomial. If
$f(x,y)\not\equiv 0,$ the degree of $f(x,y)$ is the same as the
degree of $p(x,y)$.

\begin{proposition}\label{p11}
Let $f:\ZZ^2\longrightarrow\CC$ be a polynomial, and let $\hat f
:\ZZ_+^2\longrightarrow\CC$ be the Fourier transform of the
restriction $f_{|_{\ZZ_+^2}}$. Then the function $\hat f$ has a
finite support and
$$\forall (x,y)\in\ZZ^2,\quad f(x,y)=\sum_{(m,n)\in\ZZ_+^2}\hat f(m,n)x^{[m]}y^{[n]}.$$
\end{proposition}
\begin{proof}
First, in view of \eqref{fou11} and of the fact that
$$\forall (m,n)\in\ZZ_+^2,\,\,\forall (x,y)\in\ZZ^2,
\quad m+n>\deg(f)\Longrightarrow (\delta_x^m\delta_y^n
f)(x,y)=0,$$ the function $\hat f$ has a finite support.  It
follows that
$$\sum_{(m,n)\in\ZZ_+^2}\hat f(m,n)x^{[m]}y^{[n]}$$ is, in fact, a polynomial,
which coincides with the polynomial $f(x,y)$ on $\ZZ_+^2$ and
hence on $\ZZ^2$.
\end{proof}
It follows from Proposition \ref{p11} and identity \eqref{hu2}
that if $f:\ZZ^2\longrightarrow\CC$ is a polynomial then so are
$\delta_x f$ and $\delta_yf$. A converse statement can be
formulated as follows. It will be used in the proof of Theorem
\ref{t1},
\begin{proposition}\label{p13}
Let $f,g:\ZZ^2\longrightarrow\CC$ be two polynomials, such that
\[\label{co01}(\delta_y f)(x,y)\equiv(\delta_x g)(x,y).\] Then there exists a
polynomial $h:\ZZ^2\longrightarrow\CC$ such that
\[\label{co02}(\delta_xh)(x,y)\equiv f(x,y),\quad (\delta_yh)(x,y)\equiv
g(x,y).\]
\end{proposition}
\begin{proof}
By  Proposition \ref{p11}, there exist functions $\hat f,\hat
g:\ZZ_+^2\longrightarrow\CC$  with finite support, such that
$$f(x,y)=\sum_{(m,n)\in\ZZ_+^2}\hat{f}(m,n)x^{[m]}y^{[n]},\quad
g(x,y)=\sum_{(m,n)\in\ZZ_+^2}\hat{g}(m,n)x^{[m]}y^{[n]}.$$

Then identity \eqref{co01} implies that
\[\label{use03}\forall
(m,n)\in\ZZ^2_+,\qquad (n+1)\hat{f}(m,n+1)=(m+1)\hat{g}(m+1,n).\]
 Consider
the polynomial
$$h(x,y)=\sum_{(m,n)\in\ZZ_+^2}\dfrac{\hat{f}(m,n)}{m+1}
x^{[m+1]}y^{[n]}+\sum_{n\in\ZZ_+}\dfrac{\hat{g}(0,n)}{n+1}y^{[n+1]},$$
then
$$\forall (x,y)\in\ZZ^2,\qquad(\delta_x h)(x,y)=\sum_{(m,n)\in\ZZ_+^2}\hat{f}(m,n)
x^{[m]}y^{[n]}=f(x,y).$$ On the other hand, in view of
\eqref{use03},
$$h(x,y)=\sum_{(m,n)\in\ZZ_+^2}\dfrac{\hat{g}(m,n)}{n+1}
x^{[m]}y^{[n+1]}+\sum_{n\in\ZZ_+}\dfrac{\hat{f}(m,0)}{m+1}x^{[m+1]},$$
hence
$$\forall (x,y)\in\ZZ^2,\qquad(\delta_y h)(x,y)=\sum_{(m,n)\in\ZZ_+^2}\hat{g}(m,n)
x^{[m]}y^{[n]}=g(x,y).$$  \end{proof}

\section{Discrete analytic polynomials}
\setcounter{equation}{0}
\label{sec4}
{\bf Difference operators:} It is convenient to recast Definition
\ref{maindef} in terms of the difference operators. In what
follows, each of the sets $\Omega_1,\Omega_2$ is either $\ZZ$ or
$\ZZ_+$.

\begin{theorem}
\label{maindef2}
A function $f:\Omega_1\times\Omega_2\longrightarrow\CC$ is
discrete analytic if and only if
$$\forall (x,y)\in \Omega_1\times\Omega_2,\qquad(\bar\cD f)(x,y)=0,$$
where
\[\label{COP}\bar\cD:=(1-i)\delta_x+(1+i)\delta_y+\delta_x\delta_y.\]
\end{theorem}

\begin{proof}
In view of \eqref{comm},
\begin{multline*}\forall(x,y)\in\Omega_1\times\Omega_2,\quad
\dfrac{f(x+1,y+1)-f(x,y)}{1+i}-\dfrac{f(x+1,y)-f(x,y+1)}{1-i}\\=
\dfrac{1-i}{2}(f(x+1,y+1)-f(x,y)-if(x+1,y)+if(x,y+1))\\=
\dfrac{1-i}{2}((\delta_x\delta_yf)(x,y)-2f(x,y)+
(1-i)f(x+1,y)+(1+i)f(x,y+1))\\=\dfrac{1-i}{2}((\delta_x\delta_yf)(x,y)+
(1-i)(\delta_xf)(x,y)+(1+i)(\delta_yf)(x,y)).\end{multline*}
\end{proof}

{\bf Extension:} In  view of Definition \ref{maindef}, given
$f_0\,\,:\,\,\ZZ\,\,\longrightarrow\,\,\mathbb C$ there are
infinitely discrete analytic functions $f$ on $\ZZ^2$ such that
$f(x,0)=f_0(x)$. However, the following theorem show that in the
case when $f_0$ is a polynomial, only one of these discrete
analytic extensions will be a polynomial in $x,y$. The result
itself originates with the work of Duffin \cite{MR0078441}, and
we give a new proof.

\begin{theorem}
\label{t1}
Let $p:\ZZ\longrightarrow\CC$ be a polynomial. Then there exists a
unique discrete analytic polynomial $q:\ZZ^2\longrightarrow\CC$
such that $$\forall x\in\ZZ,\qquad q(x,0)= p(x).$$
In particular, $q(x,y)\equiv 0$ if and only if $p(x)\equiv 0$. If
this is not the case,
$$\deg(q)=\deg(p).$$
\end{theorem}

\begin{lemma}\label{l1}
Let $q:\ZZ^2\longrightarrow\CC$ be a discrete analytic polynomial,
such that $q(x,0)\equiv 0.$ Then $q(x,y)\equiv 0.$
\end{lemma}
\begin{proof}
Assume the opposite, then $\deg(q)>0,$ and $q$ can be chosen so
that $\deg q$ is the smallest possible. Observe that
$(\delta_xq)(x,y)$ is also a discrete analytic polynomial, that $
(\delta_xq)(x,0)\equiv 0,$ and that $\deg(\delta_xp)<\deg(p).$
Hence $(\delta_xq)(x,y)\equiv 0.$ Since $q(x,y)$ is discrete
analytic, Definition \ref{maindef2} implies that
$$\forall (x,y)\in\ZZ^2,\qquad (\delta_yq)(x,y)=
\dfrac{i-1}{2}((1-i+\delta_y)\delta_xq)(x,y)= 0.$$ Thus
$$(\delta_xq)(x,y)\equiv(\delta_yq)(x,y)\equiv 0$$
and hence $q=\const$ - a contradiction.
\end{proof}
\begin{proof}[Proof of Theorem \ref{t1}]
The uniqueness of the polynomial $q(x,y)$ follows from Lemma
\ref{l1}. The existence in the case $p=\const$ is clear: it
suffices to set
$$q(x,y)=p(0).$$  If $p\not=\const$ we proceed by induction on
$d=\deg(p)$. According to Proposition \ref{p03}, $(\delta p)(x)$
is a polynomial and $\deg(\delta p)=d-1$. Therefore, by the
induction assumption, there is a discrete analytic polynomial
$f(x,y)$,  such that
 $$\forall x\in\ZZ,\qquad f(x,0)=(\delta p)(x)$$
 and  $\deg(f)=d-1.$ Let $g:\ZZ^2\longrightarrow\CC$ be defined by
 $$g(x,y)=if(x,y)-\dfrac{1-i}{2}(\delta_yf)(x,y),$$
 then $g(x,y)$ is also a discrete analytic polynomial,
 $\deg(g)=d-1.$ Furthermore, since $(\bar Df)(x,y)\equiv 0,$
\begin{multline*}\forall (x,y)\in\ZZ^2,\quad
(\delta_xg)(x,y)=i(\delta_xf)(x,y)-\dfrac{1-i}{2}(\delta_x\delta_yf)(x,y)\\=
i(\delta_xf)(x,y)+\dfrac{1-i}{2}((1-i)(\delta_xf)(x,y)+(1+i)(\delta_yf)(x,y))=
(\delta_yf)(x,y).\end{multline*}
 Hence, according to Proposition \ref{p13}, there exists a
 polynomial $h:\ZZ^2\longrightarrow\CC$ such that
 $$(\delta_xh)(x,y)\equiv f(x,y),\quad (\delta_y h)(x,y)\equiv
 g(x,y).$$
 Since
$$\forall (x,y)\in\ZZ^2,\qquad (\bar D
h)(x,y)=(1-i)f(x,y)+(1+i)g(x,y)+(\delta_y f)(x,y)=0,$$ the
polynomial $h(x,y)$ is discrete analytic. Finally, since
$$\forall x\in\ZZ,\qquad (\delta_x h)(x,0)=f(x,0)=(\delta p)(x),$$
$h(x,0)-p(x)$ is a constant function. Thus it suffices to set
$$q(x,y)=h(x,y)-h(0,0)+p(0)$$ to complete the proof.
\end{proof}

\section{Expandable discrete analytic functions}
\label{sec5}
In view of Theorem \ref{t1}, there exists a unique
discrete analytic polynomial $\zeta_n(x,y)$ determined by
$$\zeta_n(x,0)\equiv x^{[n]}.$$
Then (as follows from Proposition \ref{p11} and identities
\eqref{hu2},
 \eqref{comm}) $(\delta_x\zeta_n)(x,y)$ is
also a discrete analytic polynomial such that
$$(\delta_x\zeta_n)(x,0)\equiv \delta x^{[n]}\equiv nx^{[n-1]}.$$
Hence, by the uniqueness part of Theorem \ref{t1},
\[\label{hu3}(\delta_x\zeta_n)(x,y)\equiv n\zeta_{n-1}(x,y)\]
(if $n=0$, $\zeta_0(x,y)\equiv 1$ and
$(\delta_x\zeta_0)(x,y)\equiv 0$).

\begin{proposition}
\label{prop6.1} For each $(x,y)\in\ZZ^2$ the function
\[
\label{exy}
e_{x,y}(z)=(1+z)^x\left(\dfrac{1+i+iz}{1+i+z}\right)^y,
\]
is analytic (in the usual sense) in the variable $z$ in the open
unit disk $\mathbb D$, and admits the Taylor expansion
\begin{equation}
\label{exyseries} e_{x,y}(z)=\sum_{n\in\ZZ_+}\dfrac{z^n
\zeta_n(x,y)}{n!},\quad \forall z\in\DD.
\end{equation}
\end{proposition}

\begin{proof} The analyticity is clear because $x,y\in\ZZ$ and we have
$$e_{x,y}(z)=\sum_{n\in\ZZ_+}\dfrac{z^n c_n(x,y)}{n!},\quad \forall z\in\DD.$$
where
$$c_n(x,y)=\dfrac{d^n}{dz^n}e_{x,y}\Bigm|_{z=0}.$$
Since
$$e_{x+1,y}(z)-e_{x,y}(z)=ze_{x,y}(z)=\sum_{n\in\ZZ_+}\dfrac{z^{n+1}
c_{n}(x,y)}{n!},$$
$$\forall (x,y)\in\ZZ^2,\,\,\forall n\in\ZZ_+,\qquad
(\delta_xc_{n})(x,y)=nc_{n-1}(x,y)$$ (if $n=0,$ $c_0(x,y)\equiv 1$
and $(\delta_xc_0)(x,y)\equiv 0$). Similarly, since
\[
\begin{split}
\dfrac{e_{x+1,y+1}(z)-e_{x,y}(z)}{1+ i}-\dfrac{e_{x+1,y}(z)-e_{x,y+1}(z)}{1-i}&=\\
&\hspace{-5cm}=\left(
(1+z)\dfrac{1+i+iz}{1+i+z}-1\right)\dfrac{e_{x,y}(z)}{1+i}-\left(
1+z-\dfrac{1+i+iz}{1+i+z}\right)\dfrac{e_{x,y}(z)}{1-i}\\
&\hspace{-5cm}=0,
\end{split}
\]
we have
$$\forall (x,y)\in\ZZ^2,\,\,{\rm and}\,\,\forall n\in\ZZ_+,\qquad
(\bar Dc_{n})(x,y)=0.$$ Thus, for every $n\in\ZZ_+$, the function
$c_n:\ZZ^2\longrightarrow\CC$ is discrete analytic. Next, we show
by induction that
\[
\label{hu4} \forall n\in\ZZ_+,\quad c_n(x,y)\equiv\zeta_n(x,y).
\]
Indeed, for $n=0$,
$$c_0(x,y)\equiv 1\equiv\zeta_0(x,y).$$
Assume that, for some $n\in\ZZ_+$,
$$c_n(x,y)\equiv\zeta_n(x,y),$$
then, in view of \eqref{hu3},
$$(\delta_xc_{n+1}(x,y)\equiv(n+1)c_n(x,y)\equiv
(n+1)\zeta_n(x,y)\equiv(\delta_x\zeta_{n+1}(x,y),$$ hence
$$(\delta_x(\zeta_{n+1}-c_{n+1}))(x,y)\equiv 0.$$
But the functions $c_{n+1}$ and $\zeta_{n+1}$ are discrete
analytic, hence $$(\bar D(\zeta_{n+1}-c_{n+1}))(x,y)\equiv 0$$ and
$$(\delta_y(\zeta_{n+1}-c_{n+1}))(x,y)\equiv 0.$$
It follows that $$\zeta_{n+1}-c_{n+1}=\const;$$ since
$$\zeta_{n+1}(0,0)=0=c_{n+1}(0,0),$$
one concludes that
$$c_{n+1}(x,y)\equiv\zeta_{n+1}(x,y),$$
and \eqref{hu4} follows.
\end{proof}

\begin{corollary}
Let $x,n\in\mathbb Z_+$. Then,
\[
\label{value:zeta:n}
\zeta_n(x,0)=\begin{cases} x(x-1)\cdots
(x-n+1)=x^{[n]},\,\, if\quad n\le x\\
\quad0,\quad \hspace{3.82cm}if \quad n>x.
\end{cases}
\]
\end{corollary}

\begin{proof} Set $y=0$ in $e_{x,y}(z)$ in \eqref{exy}. By
\eqref{hu4} we get
\[
\label{eq:y=0}
(1+z)^x=\sum_{n\in\ZZ_+}\frac{z^n}{n!}\zeta_n(x,0).
\]
\eqref{value:zeta:n} follows by comparing the coefficients of
$z^n$ in \eqref{eq:y=0}.
\end{proof}

\begin{theorem}
\label{lim2} It holds that
$$\forall(x,y)\in\ZZ_+\times
(\ZZ\setminus\{0\}),\qquad \limsup_{n\rightarrow\infty}
\left(\dfrac{|\zeta_n(x,y)|}{n!}\right)^{1/n}=
\dfrac{1}{\sqrt{2}}.$$
\end{theorem}
\begin{proof} When $x\ge 0$ the function \eqref{exy} is analytic
in the variable $z$ in the disk centered at the origin and of
radius $\sqrt{2}$, and has a pole on the boundary of this disk.
Hence the radius of convergence of the McLaurin series is
precisely $\sqrt{2}$, and
$$\limsup_{n\rightarrow\infty}
\left(\dfrac{|\zeta_n(x,y)|}{n!}\right)^{1/n}=
\dfrac{1}{\sqrt{2}}.$$
\end{proof}

\begin{theorem}\label{p52}
Let  $g:\ZZ_+\longrightarrow\CC$ be such that for every
$(x,y)\in\ZZ_+\times\ZZ$ the series
\[\label{tay41}f(x,y)=\sum_{n\in\ZZ_+}
g(n)\zeta_n(x,y)\] converges absolutely. Then the function
 $f:\ZZ_+\times \ZZ\longrightarrow \CC$, defined by \eqref{tay41}, is discrete analytic, and it holds that
$$g(n)\equiv \hat{f_0}(n),$$ where the function
$f_0:\ZZ_+\longrightarrow\CC$ is given by
\[\label{rest}f_0(x)=f(x,0),\quad x\in\ZZ_+.\]
\end{theorem}
\begin{proof}
The discrete analyticity of $f$ follows directly from the discrete
analyticity of the polynomials $\zeta_n$. Furthermore, when $y=0$
the formula \eqref{tay41} becomes
$$f_0(x)=\sum_{n\in\ZZ_+}
g(n)x^{[n]},$$ hence, according to Proposition \ref{tay2},
$g=\hat{f_0}$.
\end{proof}

Now we can introduce the main class of functions to be considered
in this paper.

\begin{definition}
\label{5.1}
A function $f:\ZZ_+\times \ZZ\longrightarrow \CC$ is
said to be {\em expandable} if:\begin{enumerate}
\item the Fourier transform $\hat{f_0}$ of the function $f_0:\ZZ_+\longrightarrow\CC$,  given by
\eqref{rest}, satisfies the estimate
\begin{equation}
\label{esti}
\limsup_{n\rightarrow\infty}(|\hat{f_0}(n)|n!)^{1/n}<\sqrt{2};
\end{equation}
\item the function $f$ admits the representation
$$ f(x,y)=\sum_{n\in\ZZ_+}\hat{f_0}(n)\zeta_n(x,y),\quad (x,y)\in\ZZ_+\times\ZZ.$$
\end{enumerate}
\end{definition}

The class of expandable functions contains all discrete analytic
polynomials, and elements of this class are determined by their
values on the positive horizontal axis.

\begin{corollary}
\label{c6.3}
Suppose that $f_0:\ZZ_+\longrightarrow \CC$ is
rational. Then there exists a unique expandable function
$f:\ZZ_+\times \ZZ\longrightarrow \CC$ such that
$$f(x,0)\equiv f_0(x).$$
\end{corollary}
\begin{proof}
This is a consequence  of Theorem \ref{lim1}.
\end{proof}

\section{The Cauchy-Kovalevskaya product}
\label{sec6}
Theorem \ref{p52} and Corollary \ref{c6.3} allows us
to define a (partially) defined product on expandable functions,
which is everywhere defined on rational functions. This product
will be denoted by $\odot$ and called, for reasons to be explain
later in the section, the Cauchy-Kovalesvskaya product. Consider
$f_1$ and $f_2$ two expandable functions, and assume that the
Fourier transform of the pointwise product $f_1(x,0)f_2(x,0)$
satisfy \eqref{esti}. Then there exists a unique discrete analytic
expandable function $g$ such that
\[
g(x,0)=f_1(x,0)f_2(x,0)
\]
$g$ is called the  Cauchy-Kovalesvskaya product of $f_1$ and
$f_2$ and is denoted by $f_1\odot f_2$. Note that the
Cauchy-Kovalesvskaya product of two rational functions always
exist. We now  give a more formal definition of the product:

\begin{definition}
\label{formal} Let $f:\ZZ^2\longrightarrow \CC$ be a polynomial,
such that
$$f(x,0)\equiv c_0+c_1x+\dots +c_nx^n,$$
 and let
$g:\Omega\times\ZZ\longrightarrow \CC$ be given. The
Cauchy-Kovalevskaya (C-K) product of $f$ and $g$ is defined by
\begin{equation}
\label{ckdef} (g\odot f)(x,y)=(f\odot g)(x,y)=c_0g(x,y)+c_1(\cZ
g)(x,y)+\dots+c_n(\cZ^ng)(x,y).
\end{equation}
\end{definition}

We shall abbreviate this as
\[
f\odot g=f(\mathcal Z)g.
\]
The commutativity asserted in the definition is proved in the
following theorem.
\begin{theorem}\mbox{}\\
$(1)$ The restriction of the C-K product to $\mathbb Z_+$ is the
product  of the restriction.\\
$(2)$ The C-K product is the unique discrete analytic extension
corresponding to the product of the restrictions.\\
$(3)$ The C-K product is commutative, i.e.
\[
g\odot f=f\odot g
\]
for all choices of discrete analytic polynomials
\end{theorem}

\begin{proof}\mbox{}\\
$(1)$ Setting $y=0$ in \eqref{ckdef}  and taking into account the
definition of $\cZ$ we have
\[
\begin{split}
(g\odot f)(x,0)&=c_0g(x,0)+c_1(\cZ
g)(x,0)+\dots+c_n(\cZ^ng)(x,0)\\
&=c_0g(x,0)+c_1xg(x,0)+\cdots +c_nx^ng(x,0)\\
&=f(x,0)g(x,0).
\end{split}
\]
$(2)$  For an expandable function the discrete Cauchy-Riemann
equation $\overline{D}f=0$ with prescribed initial values on the
horizontal positive axis has a unique solution (see Theorem \ref{p52}).\\

$(3)$ is then clear from $(1)$ and $(2)$.
\end{proof}

We now explain the name given to this product. Recall that the
classical Cauchy-Kovalevskaya theorem concerns uniqueness of
solutions of certain partial differential equations with given
initial conditions. See for instance \cite{MR93j:26013}. In
Clifford analysis, where the pointwise product of hyperholomorphic
functions need not be hyperholomorphic, this theorem was used by
F. Sommen in \cite{MR618518} (see also \cite{bds}) to define the
product of hyperholomorphic quaternionic-valued functions in
$\mathbb R^4$ by extending the pointwise product from an
hyperplane. In the present setting of expandable functions the
discrete Cauchy-Riemann equation $\overline{D}f=0$ with
prescribed initial values on the horizontal positive axis also
has a unique solution. This is why the pointwise product on the
horizontal positive axis can be extended to a unique expandable
function, which we call the C-K product.\\

If $p$ is a discrete analytic polynomial and $f$ is an expandable
function, $p\odot f$ is the expandable function determined by
\[
(p\odot f)(x,0)=p(x,0)f(x,0).
\]
However it is not true in general that the pointwise product of
the restrictions of two expandable functions, say $f$ and $g$, is
itself the restriction of an expandable function, as is
illustrated by the example
\[
f(x,y)=g(x,y)=e_{x,y}(t),
\]
where $e_{x,y}(t)$ is defined by \eqref{exy} and $|t|>$. Indeed,
\[
e_{x,0}(t)=(1+t)^x
\]
is the restriction of an expandable function whenever
$|t|<\sqrt{2}$.  On the other hand,
$(e_{x,0}(t))^2=e_{x,0}(2t+t^2)$ will not be the restriction of
an expandable function for $|t|>\sqrt{1+\sqrt{2}}-1$ which is
strictly smaller than $\sqrt{2}$.\\

We note that the C-K product for hyperholomorphic functions was
used in  \cite{asv-cras,MR2124899,MR2240272} to define and study
rational hyperholomorphic functions, and some related reproducing
Hilbert spaces.

\begin{proposition}
For $m,n\in\ZZ_+$ and $j\in\left\{0,\ldots, m+n\right\}$, set
\[
c^{m,n}_j=\frac{\delta^j\left(x^{[m]}x^{[n]}\right)}{j!}|_{x=0}.
\]
Then,
\[
\label{equalzeta}
\zeta_m\odot\zeta_n=\sum_{j=0}^{m+n}c_j^{m,n}\zeta_j.
\]
\end{proposition}

\begin{proof}
It suffices to note that \eqref{equalzeta} holds for $y=0$,
thanks to Proposition \ref{tay2}.
\end{proof}

Systems like \eqref{equalzeta} occur in the theory of discrete
hypergroups. See \cite{MR2606478}.

\section{Rational discrete analytic functions}
\label{sec7} As we already mentioned, the pointwise product of
two discrete analytic functions need not be discrete analytic. In
the sequel of the section we define a product on discrete
analytic functions when one of the terms is a polynomial, and
show that a rational function is a quotient of discrete analytic
polynomials with respect to this product. We first need the
counterpart of multiplication by the complex variable.

\begin{definition}
\label{defmult} The {\em multiplication operator} $\cZ$ on the
class of functions $f:\Omega\times\ZZ\longrightarrow \CC$ is given
by
$$(\cZ
f)(x,y)=xf(x,y)+iy\dfrac{f(x,y+1)+f(x,y-1)}{2}.$$
\end{definition}

\begin{proposition}\label{ok}
Let $f$ be a function from $\Omega\times \ZZ$ into $\mathbb C$.
Then\\
$(1)$
\[\label{xmult}(\cZ
f)(x,0)\equiv xf(x,0).
\]
Furthermore:\\
$(2)$ If $f$ is a polynomial, then so is $\cZ f.$\\
$(3)$ If $f$ is discrete analytic, then so is $\cZ f$.
\end{proposition}
\begin{proof}
The proofs of $(1)$ and $(2)$ are clear from the definition. The
proof of $(3)$ follows from the identity \eqref{id3} in Theorem
\ref{tm34}.
\end{proof}

\begin{proposition}\label{ok2}
Let $f:\ZZ_+\times\ZZ\longrightarrow \CC$ be expandable. Then so
is $\cZ f$. In particular,
\[
\label{hu5} \forall n\in\ZZ_+,\qquad (\cZ\zeta_n)(x,y)\equiv
\zeta_{n+1}(x,y)+n\zeta_n(x,y).
\]
\end{proposition}
\begin{proof}
By Proposition \ref{ok}, for every $n\in\ZZ_+$,
$(\cZ\zeta_n)(x,y)$ is a discrete analytic polynomial. In view of
\eqref{xmult},
$$(\cZ\zeta_n)(x,0)= x\cdot x^{[n]}=x^{[n+1]}+n
x^{[n]},$$ hence formula \eqref{hu5} follows from Theorem
\ref{t1}.

Let $f_0:\ZZ_+\longrightarrow\CC$ be defined by \eqref{rest}, then
$$(\cZ f)(x,0)=
xf_0(x)=\sum_{n\in\ZZ_+}(n\hat{f_0}(n)+\hat{f_0}(n-1))x^{[n]},$$
where $\hat{f_0}(-1):=0.$ Since
$$\limsup_{n\rightarrow\infty}(|\hat{f_0}(n)|n!)^{1/n}<\sqrt{2},$$
$$\limsup_{n\rightarrow\infty}(|n\hat{f_0}(n)+\hat{f_0}(n-1)|n!)^{1/n}<\sqrt{2}.$$
Finally, since
$$f(x,y)=\sum_{n\in\ZZ_+}\hat{f_0}(n)\zeta_n(x,y),$$
where the convergence is absolute,
$$(\cZ f)(x,y)=\sum_{n\in\ZZ_+}\hat{f_0}(n)(\cZ
\zeta_n)(x,y)=\sum_{n\in\ZZ_+}(n\hat{f_0}(n)+\hat{f_0}(n-1))\zeta_n(x,y).$$
\end{proof}

From the preceeding proof we note that
\[
\zeta_1\odot \zeta_n=\cZ \zeta_n=n\zeta_n+\zeta_{n+1}.
\]

\begin{theorem}
\label{tm34}
The operators $\delta_x,\delta_y, \mathcal Z$ and
$\overline{D}$ generate a Lie algebra of linear operators on the
space of all functions from $\ZZ^2$ into $\mathbb C$. The Lie
bracket is $[A,B]=AB-BA$ and the relations on the generators are
\newcommand{\de}{\delta}
\begin{align}
\label{id1}
[\delta_x,\cZ]&=1+\delta_x,\\
[\de_y,\cZ]&= i(1+\de_y+\de_y^2),\\
\label{id3}
[\overline{D},\cZ]&=\left(\frac{1+i}{2}+\frac{i}{2}\de_y\right)\overline{D},\\
[\overline{D},\de_x]&=[\overline{D},\de_y]=[\de_x,\de_y]=0.
\label{id2}
\end{align}
\end{theorem}

\begin{proof}
The identities \eqref{id1}-\eqref{id2} can be verified by the
calculations in the proofs of the two preceding propositions.
\end{proof}

\begin{definition}
A function $f:\ZZ_+\times\ZZ\longrightarrow\CC$ is said to be a
{\em rational discrete analytic} function if $f(x,y)$ is
expandable and $f(x,0)$ is rational.
\end{definition}

\begin{theorem}
\label{t2} An expandable function
$f:\ZZ_+\times\ZZ\longrightarrow\CC$ is rational if and only if
it is a C-K quotient of discrete analytic polynomials function
that is, if and only if there exist discrete analytic polynomials
$p(x,y)$ and $q(x,y)$ such that
$$\forall x\in\ZZ_+,\qquad q(x,0)\not=0$$
and
$$(q\odot f)(x,y)\equiv p(x,y).$$
\end{theorem}
\begin{proof} Suppose first that $f$ is rational, and let $f_0(x)$
denote the restriction of $f$ to the horizontal positive axis. By
definition there exists two polynomials $p_0(x)$ and $q_0(x)$
such that $q_0(x)\not=0$ (on $\ZZ_+$) and
\[
q_0(x)f_0(x)=p_0(x),\quad x\in\ZZ_+.
\]
Let $p(x,y)$ and $q(x,y)$ denote the discrete analytic
polynomials extending $p_0(x)$ and $q_0(x)$ respectively. Then
both $p$ and $q\odot f$ are expandable functions, which coincide
on $\ZZ_+$, and therefore everywhere.\\

Conversely, let $p(x,y)$ and $q(x,y)$ be the discrete analytic
polynomials such that
$$\forall x\in\ZZ_+,\qquad q(x,0)\not=0$$and
$$(q\odot f)(x,y)\equiv p(x,y).$$
Setting $y=0$ leads to
$$\forall x\in\ZZ_+,\quad q(x,0)f(x,0)=p(x,0),$$
which ends the proof.
\end{proof}

\begin{theorem}
Let $p(x,y)$ and $q(x,y)$ be the discrete analytic polynomials
such that
$$\forall x\in\ZZ_+,\qquad q(x,0)\not=0$$
Then there is a unique expandable rational function $f$ such that
$$(q\odot f)(x,y)\equiv p(x,y).$$
\end{theorem}

\begin{proof}
Denote
$$g(x,y)=(q\odot f)(x,y).$$
According to Proposition \ref{ok2}, the function
$g:\ZZ_+\times\ZZ\longrightarrow\CC$ is expandable, and therefore
can be written as
$$g(x,y)=\sum_{n\in\ZZ_+}\hat{g_0}(n)\zeta_n(x,y),$$
where $\hat{g_0}$ is the Fourier transform of its restriction
$$g_0(x)=g(x,0).$$

In particular, by Theorem \ref{p52}, $g$ is discrete analytic. In
view of \eqref{xmult},
$$g_0(x)=q(x,0)f(x,0)=p(x,0),$$
hence, by Proposition \ref{p01},  $\hat{g_0}$ has finite support
and $g$ is a discrete analytic polynomial. In view of Theorem
\ref{t1},
$$g(x,y)\equiv p(x,y).$$
\end{proof}

\section{The $C^*$-algebra associated to expandable discrete
analytic functions}.
\label{sec8}
We denote by $\cH_{DA}$ the
reproducing kernel Hilbert space with reproducing kernel
\[
\label{rkda} K((x_1,y_1),(x_2,y_2))=
\sum_{n=0}^\infty\dfrac{\zeta_n(x_1,y_1)\zeta_n(x_2,y_2)^*}{(n!)^2},
\]
and let $e_n:=\frac{1}{n!}\zeta_n$ be the corresponding ONB in
$\cH_{DA}$. Then,

\begin{theorem}
\begin{equation}
\label{copy:shift}
\delta_xe_1=0\quad{and}\quad \delta_xe_n=e_{n-1},\quad n>1,
\end{equation}
i.e., $\delta_x$ is a copy of the backwards shift.
\end{theorem}

\begin{proof}
Using Proposition \ref{prop6.1}, we get
\begin{equation}
\label{eq1234}
\delta_xe_{x,y}(z)=e_{x+1,y}(z)-e_{x,y}(z)=ze_{x,y}(z).
\end{equation}
Substituting the expression
$e_{x,y}(z)=\sum_{n\in\ZZ_+}\frac{z^n}{n!}\zeta_n(x,y)$  into
\eqref{eq1234}, we get $\delta_x\zeta_1=0$ and
$\delta_x\zeta_n=n\zeta_{n-1}$ if $n>1$. The result
\eqref{copy:shift} follows.
\end{proof}

\begin{proposition}
In $\cH_{DA}$ we have
\[
\begin{split}
\label{deltay}
\delta_y&=\delta_x\left(I-\frac{i-1}{2}\delta_x\right)^{-1}\\
&=\sum_{n=0}^\infty\left(\frac{i-1}{2}\right)^n\delta_x^{n+1},
\end{split}
\]
where the convergence of the above series is in the operator norm
\end{proposition}

\begin{proof}
Recall that the operator $\overline{\cD}$ was defined in
\eqref{COP}. Since the elements of $\mathcal H_{DA}$ are discrete
analytic we have $\overline{\cD}=0$ in $\mathcal H_{DA}$, that is,

\[
\label{deltay1} (1-i)\delta_x+(1+i)\delta_y+\delta_x\delta_y=0,
\]
and thus
\[
\delta_y\left((1+i)I+\delta_x\right)=(1-i)\delta_x.
\]
Since $\delta_x$ is an isometry, and has in particular norm $1$,
we can solve equation \eqref{deltay1} and obtain \eqref{deltay}.
The power expansion converges in the operator norm since
\[
\|\left(\frac{i-1}{2}\right)\delta_x\|=\frac{1}{\sqrt2}<1.
\]
\end{proof}

\begin{theorem}
The $C^*$-algebra generated by $\delta_x$, or equivalently by
$\delta_x$ and $\delta_y$ is the Toeplitz $C^*$-algebra.
\end{theorem}

\begin{proof}
This follows from the preceding proposition and from
\cite{MR0213906}. Indeed it is known \cite{MR0213906} that the
Toeplitz $C^*$-algebra $\mathcal T$ is the unique $C^*$-algebra
generated by the shift. Since $\delta^*_x$ is a copy of the shift,
and $\delta_y\in C^*(\delta^*_x)$, the result follows.
\end{proof}

We now consider
\[
A={\rm Re}~\cZ=\frac{1}{2}\left(\cZ+\cZ^*\right),
\]
where $\cZ$ is defined from Definition \ref{defmult}.

\begin{theorem}
\label{tm84}
\mbox{}\\
$(i)$ The operator $A$ is essentially self-adjoint on the linear
span $\mathcal D$ of the functions $\zeta_n$, $n\in\mathbb Z_+$.\\
$(ii)$ On $\mathcal D$ it holds that
\[
[\delta_x,A]=\frac{1}{2}\left(I+\delta_x+\delta_x^2\right).
\]
$(iii)$ There exists a strongly continuous one parameter
semi-group $\alpha_t\,\,:\,\, \mathcal T\,\,\longrightarrow\,\,
\mathcal T$ such that
\[
\label{810}
(e^{itA})b(e^{-itA})=\alpha_t(b),\quad\forall
t\in\mathbb R,\,\,\forall b\in \mathcal T,
\]
where $\mathcal T$ denotes the Toeplitz $C^*$-algebra.
\end{theorem}

\begin{proof}
In $(iii)$, we denote by $e^{itA}$ the unitary one-parameter
group generated by the self-adjoint operator $A$ from $(i)$. The
matrix representation of $A$ with respect to the ONB
$e_n:=\frac{1}{n!}\zeta_n$ is:
\[
\begin{pmatrix}
 & & &   &   & & & &\\
 & & &  & 0  & & & &\\
 & & &n-2&n-1&0& & &\\
 & & &n-1&n-1&n&0& &\\
 & & &    &n &\boxed{n}&n+1&\\
 &  & &   &0 &n+1&n+1&n+2\\
 &   & &   &0&0&n+2&n+2\\
 &    & &   & & &0&\\
 \end{pmatrix}
\]
It is therefore a banded infinite matrix with terms going to
infinity linearly with $n$. It follows from \cite{MR507913} that
$A$ is essentially self-adjoint.\\

$(ii)$ By definition of $\cZ$ and $\delta_x$
\[
\begin{split}
[A,\delta_x^*](e_n)&=(A\delta_x^*-\delta_x^*A)(e_n)\\
&=\frac{1}{2}(e_n+e_{n+1}+e_{n+2})\\
                   &=\frac{1}{2}(I+\delta_x^*+\delta_x^{*2})e_n,
\end{split}
\]
and hence the result.\\

$(iii)$ We have
\begin{equation}
\label{sum} e^{itA}be^{-itA}=\sum_{n=0}^\infty
\frac{(it)^n}{n!}\left({\rm ad}~A\right)^n (b),
\end{equation}
and
\[
\label{wert} \left({\rm ad}~A\right)^{n+1} (b)=[A,\left({\rm
ad}~A\right)^n b],
\]
We verify \eqref{sum} on monomials of $\delta_x$ and $\delta_x^*$
using \eqref{wert} and induction. See \cite{MR671318,MR548006}
for more details regarding limits.
\end{proof}

The next corollary deals with a flow. For more information on
this topic, see \cite{Nelson_flows}.

\begin{corollary}
The one-parameter group $\left\{\alpha_t\right\}\subset{\rm
Aut}~(\mathcal T)$ passes to a flow on the circle group $\mathbb
T=\left\{z\in\mathbb C\,\,;\,\, |z|=1\right\}$.
\end{corollary}

\begin{proof} By \cite{MR0213906}, the Toeplitz algebra
$\mathcal T$ has a represnetation as a short exact sequence
\[
0\,\,\longrightarrow\,\,\mathcal T/\mathcal K\,\,\longrightarrow
\mathcal T\,\,\longrightarrow\,\,C(\mathbb
T)\,\,\longrightarrow\,\, 0,
\]
where $\mathcal K$ is a copy of the $C^*$-algebra of all compact
operators on $\cH_{DA}$. Hence,
\[
\mathcal T/\mathcal K\,\,\simeq\,\, C(\mathbb T).
\]
Since the left hand-side of \eqref{810} leaves $\mathcal K$
invariant, it follows that
\[
\alpha_t\,\,:\,\,\mathcal T\,\,\longrightarrow\,\,\mathcal T
\]
passes to the quotient $\mathcal T/\mathcal K\,\,\simeq\,\,
C(\mathbb T)$.
\end{proof}
\section{A reproducing kernel Hilbert space of entire functions}
As we stated in the previous section, the C-K product has the
disadvantage of not being defined for all pairs of expandable
functions. In Section \ref{next} we introduce a different product
which turns the space of expandable functions into a ring, and
consider a related reproducing kernel Hilbert space. In
preparation we introduce in the present section a reproducing
kernel Hilbert space of entire functions of a complex variable
within which the results of Section \ref{next} can be set in a
natural way.\\

To set these results in a wider setting, let us recall a few
facts on Schur analysis, that is, on the study of functions
analytic and contractive in the open unit disk. If $s_0$ is such a
function (in the sequel, we write $s_0\in\mathscr S$), the
operator of multiplication by $s_0$ is a contraction from the
Hardy space of the open unit disk $\mathbf H_2$ into itself. The
kernel
\[
\frac{1-s_0(z)s_0(w)^*}{1-zw^*}
\]
is then positive definite in the open unit $\mathbb D$, and its
associated reproducing kernel Hilbert space $H(s_0)$ was first
studied by de Branges and Rovnyak. Spaces $H(s_0)$ and their
various generalizations play an important role in linear system
theory and in operator theory. See for instance
\cite{MR2002b:47144, adrs,MR94b:47022,Dym_CBMS} for more
information. Here we replace $\mathbf H_2$ by two spaces, a space
of entire functions in the present section and a space of
discrete analytic functions
in the next section.\\

Thus, let $\cH$ be a Hilbert space, and let $\cO$ denote the space
of $\mathbf L(\cH)$-valued functions analytic at the origin, and
consider the linear operator $T$ on $\cO$ defined by
\begin{equation}
\label{opeT}
T(z^nA_n)=\dfrac{z^n}{n!}A_n,\quad A_n\in\mathbf
L(\cH).
\end{equation}
Then $T\cO$ is a space of $\mathbf L(\cH)$-valued entire
functions. The operator $T$ induces a product $\prot$ of elements
in $T\cO$ via

$$(Tf)\prot(Tg)=T(fg).$$

\begin{theorem}
Let $\cH$ be a Hilbert space and let $A$ is a bounded operator
from $\cH$ into itself. Then the $\mathbf L(\cH)$-valued entire
function
$$(I_{\cH}-zA)^{-\prot}=(\sum_{n=0}^\infty\dfrac{ z^{n}A^n}{n!})=e^{zA}$$
satisfies
\[
(I_{\cH}-zA)\prot (I_{\cH}-zA)^{-\prot}=I_{\cH},
\]
and it is the only function in $T\cO$ with this property.
\end{theorem}

\begin{proof}
This comes from the power expansion and norm estimates.
\end{proof}

Take now $\cH=\mathbb C$ and let $\mathbf H_2$ denote the Hardy
space of the unit disk. Then $T$ is a positive contractive
injection from $\mathbf H_2$ into itself. Denote by $\bH$ the
space $T\bH_2$ equipped with the range norm:
$$\forall f\in\bH_2, \qquad \|Tf\|_\bH=\|f\|_2.$$
Then $T:\bH_2\longrightarrow \bH$ is unitary, and
 $\bH$ is a reproducing kernel Hilbert space of entire functions with the
reproducing kernel
$$K_{\bH}(z,w)=\sum_{n=0}^\infty\dfrac{(zw^*)^n}{(n!)^2}.$$

\begin{proposition}
$\bH$ is the Hilbert space of entire functions such that
$$\int_\CC|f(z)|^2K_0(2|z|)dA(z)<\infty,$$
where
$$K_0(r)=\dfrac{1}{\pi}\int_\RR\exp(-r\cosh t)dt$$
is the modified Bessel function of the second kind of order $0$.
\end{proposition}

\begin{proof}
This follows from the fact that the Mellin transform of the
square of the function $\Gamma$ is the function $K_0(2\sqrt{x})$.
See for instance \cite[p. 50]{colombo} for the latter.
\end{proof}

We note that $\bH$ is contractively included in the Fock space
since the reproducing kernel of the latter is
\[
K_F(z,w)=\sum_{n=0}^\infty \frac{z^nw^{*n}}{n!},
\]
and
\[
K_F(z,w)-K_{\bH}(z,w)=\sum_{n=0}^\infty(zw^*)^n\left(
\frac{1}{n!}- \frac{1}{(n!)^2}\right)
\]
is positive definite in $\mathbb C$. See for instance
\cite[Theorem I, p. 354]{aron}, \cite{saitoh} for differences of
positive definite functions.\\

In view of Liouville's theorem, the only multipliers on $\bH$ in
the sense of the usual pointwise product are constants. The class
of multipliers in the sense of the $\prot$ product is more
interesting.

\begin{theorem}\label{s1}
A function $s\in\cO$ is a contractive $\prot$-multiplier on $\bH$
if and only if it is of the form
$$s=Ts_0,\quad s_0\in\Sch,$$ where $\Sch$ denotes the Schur class
of functions analytic and contractive in the open unit disk.
\end{theorem}

\begin{proof}
Assume first that $s\in\cO$ is a contractive $\prot$-multiplier
on $\bH$. Then $s=s\lozenge1\in\bH$ and hence $s=Ts_0$ for some
$s_0\in\bH_2$. Furthermore, let $f\in\bH_2$. Since
$s\lozenge(Tf)=T(s_0f)\in\bH$,  $s_0f\in\bH_2.$ Since
$$\|f\|_{\bH_2}=\|Tf\|_\bH\geq\|s\lozenge(Tf)\|_\bH=\|T(s_0f)\|_\bH=\|s_0f\|_{\bH_2},$$
$s_0\in\Sch.$

Conversely, if $s=Ts_0$ where $s_0\in\Sch$, and $f\in\bH_2$ then
$s\lozenge(Tf)=T(s_0f)\in\bH$ and
$$\|s\lozenge(Tf)\|_\bH=\|T(s_0f)\|_\bH=\|s_0f\|_{\bH_2}
\leq\|f\|_{\bH_2}=\|Tf\|_\bH.$$
Thus $s$ is a contractive $\prot$-multiplier on $\bH$.
\end{proof}

Let $s_0\in\mathscr S$. The operator $M_{s_0}$ of pointwise
multiplication is a contraction from $\bH_2$ into itself. The
operator range $\sqrt{I-M_{s_0}M_{s_0}^*}$ endowed with the range
norm is called the associated de Branges-Rovnyak space. We denote
it by $H(s_0)$. Similarly one can associate with $s\in T\Sch$ a
reproducing kernel
$$K_s(z,w)=\left((I-M_sM_s^*)K_\bH(\cdot,w)\right)(z),$$
where $M_s$ denotes the operator of $\prot$-multiplication
by $s$ on $\bH$. The corresponding reproducing kernel Hilbert
space is $\ran(\sqrt{I-M_sM_s^*})$ with the range norm; it will
be denoted by $H(s)$.
\begin{theorem}\label{s2}
The mapping $f\mapsto Tf$ is unitary from de Branges - Rovnyak
space $H(s_0)$ onto $H(s)$.
\end{theorem}

\begin{proof}
Since
$$M_sT=TM_{s_0},$$
$$\sqrt{I-M_sM_s^*}T=T\sqrt{I-M_{s_0}M_{s_0}^*}.$$
\end{proof}

The $H(s)$ spaces can be characterized in terms of
$\dd$-invariance, where $\dd$ is the differentiation operator:
$$\dd f = f^\prime.$$

\begin{lemma}
The operator $\dd$ is bounded on $\bH$; moreover,
\begin{equation}
\label{r0} \dd T=TR_0,
\end{equation}
where $R_0$ is the backward shift operator, and
$$\dd^*\dd=I_\bH-C^*C,\quad \dd\dd^*=I_\bH,$$
where $Cf:=f(0)$. Furthermore, the reproducing kernel of $\bH$ is
given by
$$K_\bH(z,w)=Ce^{z\dd}e^{w^*\dd^*}C^*.$$
\end{lemma}

\begin{proof} The claims follow from the definition of the
operator $T$ in \eqref{opeT}. We prove only \eqref{r0}. Let
$f\in\mathbf H_2$ with power series expansion
\[
f(z)=\sum_{n=0}^\infty a_nz^n.
\]
Then,
\[
(R_0f)(z)=\sum_{n=1}^\infty a_nz^{n-1},
\]
and therefore
\[
\begin{split}
(TR_0f)(z)&=\sum_{n=1}^\infty \frac{a_n}{(n-1)!}z^{n-1}\\
          &=\frac{d}{dz}\left(\sum_{n=0}^\infty \frac{a_n}{n!}z^n\right)\\
          &=(\dd Tf)(z).
          \end{split}
          \]
\end{proof}
\begin{theorem}
A closed subspace $H$ of $\bH$ is $\dd$-invariant if and only if
$$H=\bH\ominus M_{Ts_0}\bH,$$
where $s_0(z)$ is an inner function.
\end{theorem}
\begin{proof}
 Let $H$ be a closed subspace of $\bH$ then $H=TH_0$ where $H_0$ is a closed subspace of $\bH_2$.
$H$ is $\dd$-invariant if and only if $H_0$ is $R_0$-invariant,
which is equivalent to $\bH_2\ominus H_0$ being invariant under
multiplication by $z$. By the Beurling-Lax theorem, the last
condition holds if and only if $\bH_2\ominus H_0=M_{s_0}\bH_2,$
where $s_0$ is an inner function.
\end{proof}

\begin{theorem}
Let $s\in T\Sch$. Then $s$ admits the representation
$$s(z)=D+\int_0^zCe^{t\dd}Bdt,$$
where
$$\begin{pmatrix} \dd & B\\ C& D\end{pmatrix}:\begin{pmatrix} H(s)\\
\CC\end{pmatrix}\longrightarrow \begin{pmatrix} H(s)\\
\CC\end{pmatrix}$$ is a coisometry given by
\[
\nonumber
\begin{split}
\dd f&=f^\prime,\\
B1&=s^\prime,\\
Cf&=f(0),\\
D1&=s(0).
\end{split}
\]
\end{theorem}
\begin{proof}
Write $s=Ts_0$, where $s_0\in\Sch$. Then
$$s_0(z)=D_0+zC_0(I-zR_0)^{-1}B_0,$$
where
$$\begin{pmatrix} R_0 & B_0\\ C_0& D_0\end{pmatrix}:\begin{pmatrix} H(s_0)\\
\CC\end{pmatrix}\longrightarrow \begin{pmatrix} H(s_0)\\
\CC\end{pmatrix}$$
is a coisometry given by
\[
\nonumber
\begin{split}
R_0 f&=(f-f(0))/z,\\
B_01&=R_0s_0,\\
C_0f&=f(0),\\
D_01&=s_0(0)=s(0).
\end{split}
\]
Hence
\[
\nonumber
\begin{split}
s(z)&=D_0+\sum_{n=0}^\infty \dfrac{z^{n+1}}{(n+1)!}
C_0R_0^nB_0\\
&=D_0+\sum_{n=0}^\infty
\dfrac{z^{n+1}}{(n+1)!}C_0T^{-1}(TR_0T^{-1})^nTB_0\\
&=D+\sum_{n=0}^\infty \dfrac{z^{n+1}}{(n+1)!}C\dd^nB\\
&= D+\int_0^zCe^{t\dd}Bdt.
\end{split}
\]
\end{proof}

\begin{theorem}
Let $H$ be a Hilbert space and let
$$\begin{pmatrix}A&B\\C&D\end{pmatrix}:\begin{pmatrix}H\\ \CC
\end{pmatrix}\longrightarrow\begin{pmatrix}H\\ \CC\end{pmatrix}$$ be a
coisometry. Then the function
$$s(z)=D+\int_0^zCe^{tA}Bdt$$
is a contractive $\prot$-multiplier on $\bH$, and the
corresponding reproducing kernel is given by
$$K_s(z,w)=Ce^{zA}e^{w^*A^*}C^*.$$
\end{theorem}
\begin{proof}
Set
$$s_0(z)=D+zC(I-zA)^{-1}B,$$
then $s_0\in\Sch$ and $s=Ts_0$. Since
$$K_{s_0}(z,w)=C(I-zA)^{-1}(I-wA)^{-*}C^*,$$
the formula for $K_s(z,w)$ follows.
\end{proof}

\begin{theorem}
A reproducing kernel Hilbert space $H$ of functions in $\cO$ is of
the form $H=H(s)$ for some $s\in T\Sch$ if and only if
\begin{enumerate}
\item $H$ is $\dd$-invariant;
\item for every $f\in H$
$$\|\dd f\|_H^2\leq\|f\|_H^2-|f(0)|^2.$$
\end{enumerate}
\end{theorem}

\begin{proof}
One direction follows immediately from Theorem \ref{s2}. The
proof of the other direction is modelled after the proof of
\cite[Theorem 3.1.2, p. 85]{adrs} and is done as follows: let $H$
be $\dd$-invariant; then
 for every $f\in H$
$$Ce^{z\dd}f=f(z),$$
where $Cf=f(0).$ Hence the reproducing kernel of $H$ is given by
$$L_(z,w)=Ce^{z\dd}e^{w^*\dd^*}C^*.$$ Since
$$\dd^*\dd+C^*C\leq I,$$
there exists a coisometry

$$\begin{pmatrix} \dd & B\\ C& D\end{pmatrix}:\begin{pmatrix} H\\
\CC\end{pmatrix}\longrightarrow \begin{pmatrix} H\\
\CC\end{pmatrix}.$$ But the the function
$$s(z)=D+\int_0^zCe^{t\dd}Bdt$$
is a contractive $\prot$-multiplier and the associated kernel
$K_s(z,w)$ coincides with $L(z,w)$. Hence $H=H(s).$
\end{proof}
It is also of interest to  consider $\prot$-rational matrix
valued functions.

\begin{theorem}
\label{s5}
Tne following are equivalent:
\begin{enumerate}
\item A function $f\in T\cO$ is $\prot$-rational in the
sense that  for some  polynomial $p(z)$, not vanishing at the
origin, $p\prot f$ is also a polynomial.
\item $f(z)$ is of the form
$$f(z)=D+\int_0^zCe^{tA}Bdt$$ with $A,B,C,D$ - matrices of suitable dimensions;
\item the columns  of $\dd f$ belong to a finite-dimensional $\dd$-invariant space.
\end{enumerate}
\end{theorem}
\begin{proof}
It suffices to observe that a function $f\in T\cO$ is $\prot$-
rational if and only if it is of the form $f=Tf_0$
where$f_0\in\cO$ is rational in the usual sense.
\end{proof}

\section{A reproducing kernel Hilbert space of expandable discrete analytic function}
\label{next}
In parallel with the previous section, we introduce
the product $\boxdot$ of expandable discrete analytic functions by
\[
\zeta_n\boxdot\zeta_m=\dfrac{m!n!\zeta_{m+n}}{(m+n)!}.
\]
The advantage of this product versus the C-K one is that the
space of expandable discrete analytic functions forms a ring.
\\

Consider the linear mapping $V:z^n\mapsto\zeta_n.$ Then $VT$
maps, in particular, the space of functions analytic in a
neighborhood of the closed disk $\{z:|z|\leq1/\sqrt{2}\}$ onto
the space of expandable functions. Then $V\bH$ with the range
norm is the reproducing kernel Hilbert space $\cH_{DA}$ with the
reproducing kernel \eqref{rkda}

$$K((x_1,y_1),(x_2,y_2))=
\sum_{n=0}^\infty\dfrac{\zeta_n(x_1,y_1)\zeta_n(x_2,y_2)^*}{(n!)^2}.$$

Note that
$$V\dd=\delta_xV,\quad V(e^zA)=e_{x,y}(A).$$

Since $V:\bH\longrightarrow \cH_{DA}$ is unitary, the following
theorems are direct consequences of Theorems \ref{s1}-\ref{s5} in
the previous section. We state them here in order to emphasize
the new product.

\begin{theorem}
A closed subspace $H$ of $\cH_{DA}$ is $\delta_x$-invariant if and
only if
$$H=\cH_{DA}\ominus M_{VTs_0}\cH_{DA},$$
where $s_0(z)$ is an inner function.
\end{theorem}

\begin{theorem}
Let $s\in VT\Sch$. Then $s$ admits the representation
$$s(x,y)=D+Ce_{x,y}(\delta_x)\boxdot (\zeta_1(x,y)B),$$
where
$$\begin{pmatrix} \delta_x & B\\ C& D\end{pmatrix}:
\begin{pmatrix} H(s)\\ \CC\end{pmatrix}\longrightarrow
\begin{pmatrix} H(s)\\ \CC\end{pmatrix}$$ is a coisometry given by
\[
\nonumber
\begin{split}
B1&=\delta_xs,\\
Cf&=f(0),\\
D1&=s(0).
\end{split}
\]
\end{theorem}

\begin{theorem}
Let $H$ be a Hilbert space and let
$$\begin{pmatrix}A&B\\C&D\end{pmatrix}:\begin{pmatrix}H\\ \CC
\end{pmatrix}\longrightarrow
\begin{pmatrix}H\\ \CC\end{pmatrix}$$
be a coisometry. Then the function
$$s(z)=D+Ce_{x,y}(A)\boxdot (\zeta_1(x,y)B),$$
is a contractive $\boxdot$-multiplier on $\cH_{DA}$, and the
corresponding reproducing kernel is given by
$$K_s((x_1,y_1),(x_2,y_2))=Ce_{x_1,y_1}(A)(e_{x_2,y_2}(A))^*C^*.$$
\end{theorem}

\begin{theorem}
A reproducing kernel Hilbert space $H$ of expandable functions is
of the form $H=H(s)$ for some $s\in VT\Sch$ if and only if
\begin{enumerate}
\item $H$ is $\delta_x$-invariant;
\item for every $f\in H$
$$\|\delta_x f\|_H^2\leq\|f\|_H^2-|f(0,0)|^2.$$
\end{enumerate}
\end{theorem}

\begin{theorem}
The following are equivalent:
\begin{enumerate}
\item An expandable function $f$ is $\boxdot$-rational in the
sense that  for some discrete analytic polynomial $p(x,y)$, not
vanishing at the origin, $p\boxdot f$ is also a discrete analytic
polynomial.
\item $f(x,y)$ is of the form
$$f(x,y)=D+Ce_{x,y}(A)\boxdot (\zeta_1(x,y)B),$$
with $A,B,C,D$ - matrices of suitable dimensions, and
$\|A\|<\sqrt{2},$ and $e_{x,y}(A)$ is as in \eqref{exy}.
\item the columns  of $\delta_x f$ belong to a finite-dimensional
$\delta_x$-invariant space of expandable functions.
\end{enumerate}
\end{theorem}

\begin{theorem}
Let $s\in T\Sch$. Then $s$ admits the representation
$$s(x,y)=D+Ce_{x,y}(\delta_x)\boxdot (\zeta_1(x,y)B),$$
where
$$\begin{pmatrix} \delta_x & B\\ C& D\end{pmatrix}:
\begin{pmatrix} H(s)\\ \CC\end{pmatrix}\longrightarrow
\begin{pmatrix} H(s)\\ \CC\end{pmatrix}$$ is a coisometry given by
\[
\nonumber
\begin{split}
B1&=s^\prime,\\
Cf&=f(0),\\
D1&=s(0).
\end{split}
\]
\end{theorem}

\begin{theorem}
The following are equivalent:
\begin{enumerate}
\item An expandable function $f$ is $\boxdot$-rational in the
sense that  for some discrete analytic polynomial $p(x,y)$, not
vanishing at the origin, $p\boxdot f$ is also a discrete analytic
polynomial.
\item $f(x,y)$ is of the form
$$f(x,y)=D+Ce_{x,y}(A)\boxdot (\zeta_1(x,y)B),$$
with $A,B,C,D$ - matrices of suitable dimensions, and
$\|A\|<\sqrt{2},$ and $e_{x,y}(A)$ is as in \eqref{exy}.
\item the columns  of $\delta_x f$ belong to a finite-dimensional
$\delta_x$-invariant space.
\end{enumerate}
\end{theorem}

\bibliographystyle{plain}
%\bibliography{/users/faculty/math/dany/bib/all}
%\bibliography{all}

\begin{thebibliography}{10}

\bibitem{MR2002b:47144}
D.~Alpay.
\newblock {\em The {S}chur algorithm, reproducing kernel spaces and system
  theory}.
\newblock American Mathematical Society, Providence, RI, 2001.
\newblock Translated from the 1998 French original by Stephen S. Wilson,
  Panoramas et Synth\`eses.

\bibitem{adrs}
D.~Alpay, A.~Dijksma, J.~Rovnyak, and H.~de~Snoo.
\newblock {\em {Schur} functions, operator colligations, and reproducing kernel
  {P}ontryagin spaces}, volume~96 of {\em Operator theory: {A}dvances and
  {A}pplications}.
\newblock Birkh{\" a}user Verlag, Basel, 1997.

\bibitem{asv-cras}
D.~Alpay, M.~Shapiro, and D.~Volok.
\newblock Espaces de de {B}ranges {R}ovnyak: le cas hyper--analytique.
\newblock {\em {C}omptes {R}endus {M}ath\'ematiques}, 338:437--442, 2004.

\bibitem{MR2124899}
D.~Alpay, M.~Shapiro, and D.~Volok.
\newblock Rational hyperholomorphic functions in {$R^4$}.
\newblock {\em J. Funct. Anal.}, 221(1):122--149, 2005.

\bibitem{MR2240272}
D.~Alpay, M.~Shapiro, and D.~Volok.
\newblock Reproducing kernel spaces of series of {F}ueter polynomials.
\newblock In {\em Operator theory in Krein spaces and nonlinear eigenvalue
  problems}, volume 162 of {\em Oper. Theory Adv. Appl.}, pages 19--45.
  Birkh\"auser, Basel, 2006.

\bibitem{aron}
N.~Aronszajn.
\newblock Theory of reproducing kernels.
\newblock {\em Trans. Amer. Math. Soc.}, 68:337--404, 1950.

\bibitem{MR94b:47022}
M.~Bakonyi and T.~Constantinescu.
\newblock {\em Schur's algorithm and several applications}, volume 261 of {\em
  Pitman Research Notes in Mathematics Series}.
\newblock Longman Scientific \& Technical, Harlow, 1992.

\bibitem{bgk1}
H.~Bart, I.~Gohberg, and M.A. Kaashoek.
\newblock {\em Minimal factorization of matrix and operator functions},
  volume~1 of {\em {Operator {T}heory: {A}dvances and {A}pplications}}.
\newblock Birkh{\" a}user Verlag, Basel, 1979.

\bibitem{bds}
F.~Brackx, R.~Delanghe, and F.~Sommen.
\newblock {\em Clifford analysis}, volume~76.
\newblock Pitman research notes, 1982.

\bibitem{MR671318}
Ola Bratteli and Palle E.~T. J{\o}rgensen.
\newblock Unbounded derivations tangential to compact groups of automorphisms.
\newblock {\em J. Funct. Anal.}, 48(1):107--133, 1982.

\bibitem{MR0213906}
L.~A. Coburn.
\newblock The {$C^{\ast} $}-algebra generated by an isometry.
\newblock {\em Bull. Amer. Math. Soc.}, 73:722--726, 1967.

\bibitem{colombo}
Serge Colombo.
\newblock {\em Les transformations de {M}ellin et de {H}ankel: {A}pplications
  \`a la physique math\'ematique}.
\newblock Monographies du Centre d'\'Etudes Math\'ematiques en vue des
  Applications: B.--M\'ethodes de Calcul. Centre National de la Recherche
  Scientifique, Paris, 1959.

\bibitem{MR0078441}
R.~J. Duffin.
\newblock Basic properties of discrete analytic functions.
\newblock {\em Duke Math. J.}, 23:335--363, 1956.

\bibitem{Dym_CBMS}
H.~Dym.
\newblock {\em ${J}$--contractive matrix functions, reproducing kernel
  {H}ilbert spaces and interpolation}.
\newblock Published for the Conference Board of the Mathematical Sciences,
  Washington, DC, 1989.

\bibitem{MR0013411}
Jacqueline Ferrand.
\newblock Fonctions pr\'eharmoniques et fonctions pr\'eholomorphes.
\newblock {\em Bull. Sci. Math. (2)}, 68:152--180, 1944.

\bibitem{MR0361111}
C.~J. Harman.
\newblock A note on a discrete analytic function.
\newblock {\em Bull. Austral. Math. Soc.}, 10:123--134, 1974.

\bibitem{MR0052526}
Rufus Isaacs.
\newblock Monodiffric functions. {C}onstruction and applications of conformal
  maps.
\newblock In {\em Proceedings of a symposium}, National Bureau of Standards,
  Appl. Math. Ser., No. 18, pages 257--266, Washington, D. C., 1952. U. S.
  Government Printing Office.

\bibitem{MR507913}
Palle E.~T. J{\o}rgensen.
\newblock Essential self-adjointness of semibounded operators.
\newblock {\em Math. Ann.}, 237(2):187--192, 1978.

\bibitem{MR0255260}
R.~E. Kalman, P.~L. Falb, and M.~A. Arbib.
\newblock {\em Topics in mathematical system theory}.
\newblock McGraw-Hill Book Co., New York, 1969.

\bibitem{MR93j:26013}
S.~G. Krantz and H.~P. Parks.
\newblock {\em A primer of real analytic functions}, volume~4 of {\em Basler
  Lehrb\"ucher [Basel Textbooks]}.
\newblock Birkh\"auser Verlag, Basel, 1992.

\bibitem{MR2606478}
Rupert Lasser and Eva Perreiter.
\newblock Homomorphisms of {$l^1$}-algebras on signed polynomial hypergroups.
\newblock {\em Banach J. Math. Anal.}, 4(2):1--10, 2010.

\bibitem{MR1824204}
Christian Mercat.
\newblock Discrete {R}iemann surfaces and the {I}sing model.
\newblock {\em Comm. Math. Phys.}, 218(1):177--216, 2001.

\bibitem{Nelson_flows}
Edward Nelson.
\newblock {\em Topics in dynamics. {I}: {F}lows}.
\newblock Mathematical Notes. Princeton University Press, Princeton, N.J.,
  1969.

\bibitem{MR548006}
Gert~K. Pedersen.
\newblock {\em {$C^{\ast} $}-algebras and their automorphism groups}, volume~14
  of {\em London Mathematical Society Monographs}.
\newblock Academic Press Inc. [Harcourt Brace Jovanovich Publishers], London,
  1979.

\bibitem{saitoh}
S.~Saitoh.
\newblock {\em Theory of reproducing kernels and its applications}, volume 189.
\newblock Longman scientific and technical, 1988.

\bibitem{MR618518}
F.~Sommen.
\newblock A product and an exponential function in hypercomplex function
  theory.
\newblock {\em Applicable Anal.}, 12(1):13--26, 1981.

\end{thebibliography}
\def\cprime{$'$} \def\lfhook#1{\setbox0=\hbox{#1}{\ooalign{\hidewidth
  \lower1.5ex\hbox{'}\hidewidth\crcr\unhbox0}}} \def\cprime{$'$}
  \def\cprime{$'$} \def\cprime{$'$} \def\cprime{$'$} \def\cprime{$'$}

\end{document}